\documentclass[3p]{elsarticle}

\usepackage{epsfig}
\usepackage{color}
\usepackage{subeqnarray,multirow}
\usepackage{enumitem}
\usepackage{textcomp}
\usepackage{booktabs}
\usepackage{algorithm}
\usepackage{algpseudocode}

\usepackage{algorithmicx}

\usepackage{mathtools}
\usepackage{amsthm}
\usepackage{amsmath}
\usepackage{mathtools,amsfonts,cases}
\usepackage{textcomp}
\usepackage{subcaption}
\newcommand{\RNum}[1]{\uppercase\expandafter{\romannumeral #1\relax}}
\usepackage{xcolor}
\biboptions{sort&compress}
\theoremstyle{plain}
\newtheorem{theorem}{Theorem}

\theoremstyle{definition}

\theoremstyle{remark}

\theoremstyle{example}
\newtheorem{exam}{Example}
\newtheorem{corollary}{Corollary}

\makeatletter
\def\ps@pprintTitle{%
  \let\@oddhead\@empty
  \let\@evenhead\@empty
  \let\@oddfoot\@empty
  \let\@evenfoot\@empty}
\makeatother
\begin{document}
\begin{frontmatter}
\title{Adaptation and development of super schemes for unconstrained
optimization problems}
\author[NUM,SHU]{Tugal Zhanlav}
\ead{tzhanlav@yahoo.com}
\author[GMIT]{Lkhamsuren Altangerel}\ead{altangerel@gmit.edu.mn}
\author[SHU]{Khuder Otgondorj \corref{Otgoo}} \ead{
otgondorj@gmail.com}
\address[NUM]{ Institute of Mathematics and Digital Technology, Mongolian Academy of Sciences, 13330 Mongolia}
\address[GMIT]{Faculty of Mechanical and Electrical Engineering, German-Mongolian Institute for Resources and Technology}
\address[SHU]{School of Applied Sciences, Mongolian University of Science and
Technology, 14191 Mongolia} \cortext[Otgoo]{Corresponding author}
\begin{abstract}
In this paper, we propose a class of super-schemes for efficiently solving nonlinear unconstrained optimization problems. The proposed approach introduces two novel choices of step-size parameters, leading to efficient descent directions without requiring second-order information.
We develop one-step, two-step, and three-step iterative schemes (denoted by SS1, SS2, and SS3) and establish that these methods achieve higher-order convergence of orders two, four, and six, respectively. Despite their high convergence rates, the computational complexity of the proposed methods remains comparable to existing gradient-based methods, with a cost of $\mathcal{O}(n^2)$ per iteration.
The proposed methods are simple to implement and do not require complicated line-search procedures. Their effectiveness is demonstrated through extensive numerical experiments on a wide range of problems, including large-scale and ill-conditioned cases. The results show that the proposed methods significantly outperform classical methods, such as the Barzilai–Borwein method and other gradient-based approaches, in terms of iteration count and computational efficiency.

Finally, the numerical results are consistent with the theoretical analysis, confirming the stability of the proposed schemes for test optimization problems.
\end{abstract}
\end{frontmatter}
\section{Introduction}
We consider nonlinear unconstrained optimization problems
\begin{equation}\label{UNCO-1}
\min f(x), \quad x \in \mathbb{R}^n,
\end{equation}
where $f:\mathbb{R}^n \to \mathbb{R}$ is a convex and twice continuously differentiable function.
There are many iterative algorithms for solving \eqref{UNCO-1} that have a specific form
\begin{equation}\label{UNCO-2}
x_{k+1} = x_k + \alpha_k d_k, \quad k=0, 1,\dots
\end{equation}
Here, $x_0$ is the initial guess for the solution $x^\ast$ of problem \eqref{UNCO-1},  $d_k$ is the
descent direction and $\alpha_k$ is the line search or step-size.
The pair $(\alpha_k,d_k)$ plays a key role in \eqref{UNCO-2} and is obtained using various
approaches and techniques such as steepest descent, Newton direction,
conjugate gradient and so on.
The steepest descent (SD) method exhibits convergence near the minimum point due to its tendency to zigzag and is highly sensitive to ill-conditioned problems. Modifying the step length can help reduce or eliminate zig-zagging. Another extensively studied approach to reduce zig-zagging is the inclusion of a relaxation or damping parameter, as demonstrated in \cite{UNCO-1,Hassan2001,Mac23} and references therein.
In terms of line search,  exact and inexact adaptive step-sizes are often used, along with monotone line search and other techniques. For more details,
we refer to the comprehensive survey \cite{UNCO-4} and \cite{UNCO-1,UNCO-3,Hassan2001,UNCO-12,UNCO-1222,UNCO-13,Sulaiman2018,Li19,Ov20,Zho21,Hes22,Mac23,Joy24}. The success of iteration \eqref{UNCO-2} essentially depends on the suitable choices of
$\alpha_k$ and $d_k$.

In order to find the solution to \eqref{UNCO-1}, we need to consider the following
system of nonlinear gradient equations
\begin{equation}\label{UNCO-3}
g(x) = 0,
\end{equation}
where $g(x) = \nabla f(x)$. We assume that the gradient $g(x)$ is sufficiently differentiable.
The equation \eqref{UNCO-3} is known as the first necessary optimality condition for \eqref{UNCO-1}.
For convex problems, \eqref{UNCO-3} is not only necessary but also sufficient for both local and global optimality \cite{UNCO-1,UNCO-4,UNCO-12}.

Note that all gradient based methods have a linear convergence rate. Therefore,  developing efficient higher order methods is a challenging task.
The aim of this paper is to propose effective iterative algorithms for solving
systems of nonlinear equations \eqref{UNCO-3}.
\section{Super-schemes for nonlinear unconstrained optimization problems}\label{UNCOSec2}
Based on the recent study in \cite{UNCO-10}, we will now consider the following schemes:
\begin{equation}\label{UNCO-20}
\begin{aligned}
y_{k} &= x_k-\alpha_k g(x_k),\\
x_{k+1} &= y_k-\alpha_kT_kg(y_k),
\end{aligned}
\end{equation}
and
\begin{equation}\label{UNCO-30}
\begin{aligned}
y_k &= x_k-\alpha_k g(x_k),\\
z_k &= y_k-\alpha_kT_kg(y_k),\\
x_{k+1} & =z_k-\alpha_kT_kg(z_k),
\end{aligned}
\end{equation}
where
\begin{equation}\label{UNCO-7}
\alpha_k=
\frac{( g(w_k)-g(x_k))^T g(x_k)}
{\|g(w_k)-g(x_k)\|^2},
\end{equation}
or equivalent 
\begin{equation}\label{UNCO-7new}
\alpha_k=
\frac{\|g(x_k)\|^2}
{g(x_k)^{T}(g(w_k)-g(x_k))},
\end{equation}
and
\begin{equation}\label{UNCO-7add}
T_k={\bf1}+\Theta_k+r_k+O(h^2),
\end{equation}
\begin{equation}\label{UNCO-8add}
 \Theta_k=\frac{g(y_k)}{g(x_k)},\quad
r_k=\frac{g(y_k)}{g(w_k)},\quad
 \qquad w_k=x_k+g(x_k),\quad {\bf1}=\Big(1, 1,\dots,1\Big)^T\in \mathbb{R}^n.
\end{equation}
 The parameter $T_k$ can be rewritten as follows:
\begin{equation}\label{UNCOTk-20}
T_k = 1 + \frac{g(x_k)^T g(y_k)}{\|g(x_k)\|^2} + \frac{g(w_k)^T g(y_k)}{\|g(w_k)\|^2}.
\end{equation}
The fourth and sixth-order convergence of iterations \eqref{UNCO-20} and \eqref{UNCO-30} occur under the assumptions that $g(x)$ is sufficiently differentiable, $g'(x)$ is nonsingular at $x^*$ and $x_0$ is close enough to $x^*$. When $x_0$ is far from $x^*$, using optimal step size is often necessary to attain global convergence.
In this scenario, we can rephrase the first sub-step in \eqref{UNCO-20} and \eqref{UNCO-30} as an independent gradient-based iteration, represented by \eqref{UNCO-2}, i.e.,
\begin{equation}\label{UNCO-6}
x_{k+1}=x_k+\alpha_k d_k,
\end{equation}
where
\begin{equation}\label{UNCO-8}
d_k=-g(x_k),
\end{equation}
and the step-size $\alpha_k$ is defined by \eqref{UNCO-7}.
The descent directions given by \eqref{UNCO-8} satisfies the descent direction condition
\begin{equation*}
g(x_k)^Td_k<0.
\end{equation*}
The main advantage of iteration \eqref{UNCO-6} is that the step-size $\alpha_k$ is determined by formula \eqref{UNCO-7}, which eliminates the need for line search algorithms in this case. Additionally, if $g(x)$ is a monotone function, then the following inequality holds:
\begin{equation*}
(g(x)-g(y))^T (x-y) > 0, \quad \forall x \neq y.
\end{equation*}
By choosing  $x=w_k$ and $y=x_k$, we have
\begin{equation}\label{UNCO-8extra}
(g(w_k)-g(x_k))^T g(x_k) > 0,
\end{equation}
This means that $\alpha_k>0$ if  $\alpha_k$ is given by the formula \eqref{UNCO-7}.

It is easy to verify that the step-size $\alpha_k$ defined by \eqref{UNCO-7}
approximately minimizes
\begin{equation}\label{UNCO-12}
\alpha_k = \arg\min_{\alpha}\,
\|g(x_k)-\alpha (g(w_k)-g(x_k))\|^2 .
\end{equation}
Indeed, considering
\begin{equation}\label{UNCO-13}
g(x_k+\alpha d_k)=g(x_k-\alpha g(x_k)).
\end{equation}
and using the Taylor expansion of $g(x)$ at point $x_k$ along with the approximate formula
\[
g'(x_k)\approx \frac{g(w_k)-g(x_k)}{g(x_k)},
\]
we can rewrite \eqref{UNCO-13} as
\begin{equation}\label{UNCO-13extra}
g(x_k-\alpha g(x_k))
= g(x_k)-\alpha g'(x_k)g(x_k)+O(h^2)
= g(x_k)-\alpha (g(w_k)-g(x_k))+O(h^2),
\end{equation}
where $h=g(x_k)$.
\begin{theorem}\label{UCOth2}
Let $g:\mathbb{R}^n \to \mathbb{R}^n$ be a continuous, sufficiently differentiable and monotone function.
Then, for any initial point $x_0 \in \mathbb{R}^n$, the sequence $\{x_k\}$ generated by iteration \eqref{UNCO-6} globally converges to the solution $x^*$ of equation \eqref{UNCO-3}.
\end{theorem}
\begin{proof}
From \eqref{UNCO-12}, \eqref{UNCO-13extra}, we obtain the quadratic approximation
\begin{equation}\label{UNCO-P1}
\|g(x_k-\alpha g(x_k))\|^2
= \|g(x_k)\|^2
-2\alpha\, ( g(x_k), g(w_k)-g(x_k))
+\alpha^2 \|g(w_k)-g(x_k)\|^2,
\end{equation}
where higher-order small term is neglected.

The right-hand side of \eqref{UNCO-P1} is a quadratic function of $\alpha$. The minimizer of this function is given by
\begin{equation} \label{UNCO-15}
\alpha_k=
\frac{(g(w_k)-g(x_k))^T g(x_k)}
{\|g(w_k)-g(x_k)\|^2},
\end{equation}
which coincides with \eqref{UNCO-7}. Substituting $\alpha=\alpha_k$ into \eqref{UNCO-P1}, we obtain
\begin{equation}\label{UNCO-P2}
\|g(x_k-\alpha_k g(x_k))\|^2
= \|g(x_k)\|^2
- \frac{\big((g(w_k)-g(x_k))^T g(x_k)\big)^2}
{\|g(w_k)-g(x_k)\|^2}.
\end{equation}
By virtue of \eqref{UNCO-8extra} and from \eqref{UNCO-P2} we can deduce that
\begin{equation}\label{UNCO-P5}
\|g(x_{k+1})\| < \|g(x_k)\|,\quad k=0, 1,\dots
\end{equation}
In other words, since the sequence $\{\|g(x_k)\|\}$ is monotonically decreasing and bounded below by zero, it is convergent. Specifically, for any $x_0 \in \mathbb{R}^n$, we have $\|g(x_k) \|\to0$ as $k\to\infty$. Due to the continuity property of $g(x)$, we can conclude that $x_k\to x^*$ as $k\to\infty.$
\end{proof}
The convergence of iteration \eqref{UNCO-6}, like other gradient descent methods, slows down when the problem is more ill-conditioned \cite{UNCO-4,UNCO-12}. Therefore, it is necessary to use an accelerating technique. Fortunately, in our iterations  \eqref{UNCO-20} and \eqref{UNCO-30}, the second and third sub-steps act as accelerators.
\begin{corollary}
The fourth and sixth order convergence of iterations \eqref{UNCO-20} and \eqref{UNCO-30} immediately follows from Theorem \ref{UCOth2}.
\end{corollary}
We refer to the step-size defined by \eqref{UNCO-7} as optimal because of \eqref{UNCO-12}.
It should be pointed out that our choices in \eqref{UNCO-7} and \eqref{UNCO-7new} are similar to the two
step-sizes used in the Barzilai and Borwein methods \cite{UNCO-1}, specifically,
\begin{equation}\label{UNCO-19}
\alpha_k=\frac{s_k^T y_k}{\|y_k\|^2}
\end{equation}
\begin{equation}\label{UNCO-19n}
\alpha_k=\frac{\|s_k\|^2}{s_k^T y_k},
\end{equation}
where $y_k=g_{k}-g_{k-1}$ and $s_k=x_{k}-x_{k-1}$.
These choices ensure super-linear convergence when the objective function is a convex
quadratic function of two variables \cite{UNCO-4}. In practice, various modifications and improvements of the BB-choices \eqref{UNCO-19} and \eqref{UNCO-19n} are often, see \cite{Ov20,Zho21} and references therein.

Our choice in \eqref{UNCO-7} is seen as a specific implementation of \eqref{UNCO-19} with
$s_k=w_k-x_k=g(x_k)$.
When the initial guess is far from the solution, a line search is
frequently necessary to attain global convergence.
Our optimal step-size ensures the quadratic convergence of the method \eqref{UNCO-6} and is considered a significant descovery in the line search procedure.
It is straightforward to transition from \eqref{UNCO-20} and \eqref{UNCO-30} to their scalar coefficient
variants \cite{UNCO-11}. We present the final results here.
\begin{equation}\label{UNCO-32}
\begin{aligned}
y_{k} &= x_k+\alpha_k d_k,\\
x_{k+1}&= x_k+\alpha_k(1+\beta_k)d_k,
\end{aligned}
\end{equation}
and
\begin{equation}\label{UNCO-34}
\begin{aligned}
y_k &= x_k+\alpha_k d_k,\\
z_k &= x_k+\alpha_k(1+\beta_k)d_k,\\
x_{k+1} &= x_k+\alpha_k(1+\beta_k+\gamma_k\beta_k)d_k,
\end{aligned}
\end{equation}
where
\begin{equation}\label{UNCO-36}
\beta_k=\|g(y_k)\|^2
\left(
\frac{1}{g(y_k)^Tg(x_k)}
+\frac{1}{\|g(x_k)\|^2}
+\frac{1}{g(x_k)^Tg(w_k)}
\right),
\end{equation}
and
\begin{equation}\label{UNCO-38}
\gamma_k=\frac{\|g(z_k)\|^2}{g(z_k)^T g(y_k)}.
\end{equation}
Choosing $d_k=-\alpha_kg(x_k)$ 
we obtain that
\begin{equation}\label{UNCO-44}
\begin{aligned}
y_{k} &= x_k+d_k,\\
x_{k+1} &= x_k+(1+\beta_k)d_k,
\end{aligned}
\end{equation}
and
\begin{equation}\label{UNCO-46}
\begin{aligned}
y_k &= x_k+d_k,\\
z_k &= x_k+(1+\beta_k)d_k,\\
x_{k+1}&= x_k+(1+\beta_k+\beta_k\gamma_k)d_k.
\end{aligned}
\end{equation}
instead of
\eqref{UNCO-32}, \eqref{UNCO-34}.
In addition to  \eqref{UNCO-1}, there is another optimization problem known as the convex quadratic minimization problem
\begin{equation}\label{UNCOlin-1}
\min f(x)=\frac{1}{2}x^TAx-x^Tb,
\end{equation}
where $b\in \mathbb{R}^n$ and $A\in \mathbb{R}^{n\times n}$ is symmetric positive define matrix. It is well known that the problem \eqref{UNCOlin-1} is equivalent to linear system
\begin{equation}\label{UNCOlin-2}
Ax-b=0.
\end{equation}
In this case the equation \eqref{UNCO-3} transforms into \eqref{UNCOlin-2}. Our proposed schemes \eqref{UNCO-6}, \eqref{UNCO-20} and \eqref{UNCO-30}, along with their scalar coefficient variants, work well.
The proposed one-step method \eqref{UNCO-6} is implemented as follows:
\begin{algorithm}[H]
\caption{Proposed SS1}
\label{alg:UNCO}
\begin{algorithmic}[1]
\State \textbf{Input:} Initial point $x_0 \in \mathbb{R}^n$,  tolerance $\varepsilon > 0$
\State \textbf{Initialize:} $k = 0$
\While{$\|g(x_k)\| \ge \varepsilon$}
    \State Compute $w_k = x_k +\, g(x_k)$
    \State Compute step size by formula \eqref{UNCO-7}
\State \textbf{(or )}
    formula \eqref{UNCO-7new}
 \State Compute direction $d_k = - g(x_k)$
    \State Update iterate $x_{k+1} = x_k + \alpha_kd_k$
    \State $k \leftarrow k + 1$
\EndWhile
\State \textbf{Output:} $x_k$
\end{algorithmic}
\end{algorithm}
\section{Numerical experiments} \label{Dsec5}
In this section, we present the numerical results obtained from implementing the proposed methods in MATLAB R2024a for a set of test functions selected from \cite{UNCO-12,UNCO-1222}.
The proposed methods are compared with various gradient-based methods introduced in \cite{UNCO-4, UNCO-12, Hes22}, including the Barzilai–Borwein (BB) method and the Neculai Andrei (NA) method.
These experiments were conducted using an Intel Core i5-4590 processor running at 3.30 GHz with 4096 MB of RAM
memory. 
The iteration is terminated when either $|g(x_k)| \leq 10^{-6}$ or after reaching the maximum number of iterations (2000).
In our experiments, we denote the proposed one-step method \eqref{UNCO-6}, two-step method \eqref{UNCO-20}, and three-step method \eqref{UNCO-30} as SS1, SS2, and SS3, respectively. The parameter $\alpha_k$ for SS1, SS2, and SS3 is selected according to \eqref{UNCO-7}. It is important to note that when $\alpha_k$ is chosen based on \eqref{UNCO-7new}, all three methods yield similar results. The parameter $T_k$ is also utilized for SS2 and SS3, as defined in \eqref{UNCOTk-20}. Additionally, for the BB method \cite{UNCO-1}, we choose the parameter $\alpha_k$ as specified in \eqref{UNCO-19n}.

Tables \ref{UNCONtab-1}–\ref{UNCONtab-8}, clearly demonstrate the effectiveness of the proposed methods on the best functions considered. Specifically, the SS2 and SS3 methods show faster convergence, requiring fewer iterations than the competing methods in most cases.
Furthermore, to assess the performance of our methods, we compare our results with those presented in \cite{Ov20} for convex quadratic minimization problems. Table \ref{UNCONtab-9} provides this comparison.
In the tables, ``--'' indicates that the number of iterations exceeded the maximum allowed. We compute the approximate computational order of convergence (ACOC) of the proposed methods using the formula provided in \cite{Joy24, UNCO-10}:
\begin{equation}\label{UNCOC1223}
\rho_k = \frac{
\log \left( \| g(x_{k+1}) \| / \| g(x_{k}) \| \right)
}{
\log \left( \| g(x_{k}) \| / \| g(x_{k-1}) \| \right)
}.
\end{equation}
Formula \eqref{UNCOC1223} is commonly used to compute the approximate computational order of convergence (ACOC) of iterative methods for solving systems of nonlinear equations. The results are reported in Table \ref{UNCONtab-10} obtained using the formula \eqref{UNCOC1223} for the gradient system (refer to equation \eqref{UNCONEx-st33}) corresponding to Example \ref{UNCONEx-1}. We present the ACOC for a simple test problem; however, for more challenging cases such as ill-conditioned problems, calculating the ACOC using formula \eqref{UNCOC1223} becomes challenging.
\begin{exam} \label{UNCONEx-1}
\begin{equation*}
f(x) = \sum_{i=1}^{n} \left(e^{x_i} - x_i\right),
\end{equation*}
Initial approximation: $x_0 = (1, 1, \ldots, 1, \ldots, 1)^T $
\end{exam}
\begin{exam} \label{UNCONEx-2}
\begin{equation*}
\begin{aligned}
f(x) &= \big((5 - 3x_1 - x_1^2)x_1 - 3x_2 + 1\big)^2 \\
&\quad + \sum_{i=2}^{n-1}
\big((5 - 3x_i - x_i^2)x_i - x_{i-1} - 3x_{i+1} + 1\big)^2 \\
&\quad + \big((5 - 3x_n - x_n^2)x_n - x_{n-1} + 1\big)^2,
\end{aligned}
\end{equation*}
Initial approximation: $x_0 = (-0.8, -0.8, \dots, -0.8)^T .$
\end{exam}
\begin{exam} \label{UNCONEx-3}
\begin{equation*}
f(x) = \sum_{i=1}^{n-1} \left[ (x_{i+1} - x_i^2)^2 + (1 - x_i)^2 \right],
\end{equation*}
Initial approximation: $x_0 = (-1.2, -1.2,  \ldots, -1.2)^T $
\end{exam}
\begin{exam} \label{UNCONEx-4}
\begin{equation*}
f(x) = \sum_{i=1}^{n} \frac{i}{10} (\exp(x_i) - x_i),
\end{equation*}
\end{exam}
Initial approximation: $x_0 = (0.3, 0.3,  \ldots, 0.3)^T $
\begin{exam} \label{UNCONEx-5}
\begin{equation*}
f(x) = \sum_{i=1}^{n-1} \left( x_{i+1} - x_i^3 \right)^2 + (1 - x_i)^2
\end{equation*}
\end{exam}
Initial approximation: $x_0 = (-1, 2, 1  \ldots, -1,2,1)^T $
\begin{exam} \label{UNCONEx-6}
\begin{equation*}
f(x) = \sum_{i=1}^{n/2} \Bigg[
\big(x_{2i-1}^2 + x_{2i}^2 + x_{2i-1} x_{2i}\big)^2
+ \sin^2(x_{2i-1}) + \cos^2(x_{2i})
\Bigg],
\end{equation*}
Initial approximation: $x_0 = (3, 0.1, \dots, 3, 0.1)^T$
\end{exam}
\begin{exam} \label{UNCONEx-8} We consider the problem of minimizing the following function:
\begin{equation*}
f(x) = \frac{1}{2} \langle x, Ax \rangle - \langle b, x \rangle,
\end{equation*}
where $A =diag\left(1,2,3\dots,100\right)$.
$b = (1,1,\dots,1).$
Initial approximation: $x_0 = (0, 0, \dots,0,0)^T$
\end{exam}
\begin{exam} \label{UNCONEx-9}
$A$ is a tridiagonal matrix whose entries are defined by
\begin{equation}
a_{i,i} = \frac{2}{h^2}; \quad a_{i,i-1} = -\frac{1}{h^2} \text{ if } i \neq 1; \quad a_{i,i+1} = -\frac{1}{h^2} \text{ if } i \neq n,
\end{equation}
for all $i \in \{1, \dots, n\}$, where $h = 11/n$ and $n$ varies in $n \in \{500, 1000, 1500, 2000\}$, this example was taken from \cite{Li19}. The optimal solution $x^*$ was first generated by the following MATLAB command $x^* = -10 + 20 * \text{rand}(n, 1)$ and then we set $b = Ax^*$. These types of problems often arise in the numerical solution of two-point boundary value problems, and can result in very ill-conditioned matrices. 
\end{exam}
\begin{table}[th!]
\centering \caption{Iteration number for Example \ref{UNCONEx-1}} 
\begin{tabular}{lllccc}
\hline
$n$ & SS1   &SS2&SS3&BB&NA\\ \hline
1000&7&4&4&7&6 \\ \hline
2000&7&4&4&7&6\\ \hline
5000&7&5&5&7&7\\ \hline
10000&7&5&4&7&6\\ \hline
50000&7&5&4&7&7\\ \hline
100000&7&5&4&7&7\\ \hline
\end{tabular}\label{UNCONtab-1}
\end{table}
\begin{table}[th!]
\centering \caption{Iteration number for Example \ref{UNCONEx-2}} 
\begin{tabular}{lllccc}
\hline
$n$ & SS1   &SS2&SS3&BB&NA\\  \hline
1000&72&37&18&27&42 \\ \hline
2000&31&39&19&24&48\\ \hline
5000&28&36&18&24&22\\ \hline
10000&25&34&18&24&22\\ \hline
20000&29&30&18&32&23\\ \hline
50000&41&21&18&24&25\\ \hline
\end{tabular}\label{UNCONtab-2}
\end{table}
\begin{table}[th!]
\centering \caption{Iteration number for Example \ref{UNCONEx-3}} 
\begin{tabular}{lllccc}
\hline
$n$ & SS1   &SS2&SS3&BB&NA\\ \hline
1000&304&60&54&75&65 \\ \hline
2000&296&66&44&76&65\\ \hline
5000&293&60&45&87&88\\ \hline
10000&302&62&48&65&63\\ \hline
20000&326&60&109&74&84\\ \hline
50000&361&60&53&94&81\\ \hline
\end{tabular}\label{UNCONtab-3}
\end{table}
\begin{table}[th!]
\centering \caption{Iteration number for Example \ref{UNCONEx-4}} 
\begin{tabular}{lllcc}
\hline
$n$ & SS1   &SS2&BB&NA\\ \hline
1000&494&294&276&261 \\ \hline
2000&925&391&466&381\\ \hline
5000&1902&715&689&685\\ \hline
10000&-&708&1093&1982\\ \hline
\end{tabular}\label{UNCONtab-4}
\end{table}
\begin{table}[th!]
\centering \caption{Iteration number for Example \ref{UNCONEx-5}} 
\begin{tabular}{lllcc}
\hline
$n$ & SS1   &SS2&BB&NA\\ \hline
1000&1209&182&232&197 \\ \hline
2000&1134&183&222&191\\ \hline
5000&1019&175&233&198\\ \hline
10000&1063&185&201&204\\ \hline
20000&951&180&215&226\\ \hline
50000&1009&168&214&220\\ \hline
\end{tabular}\label{UNCONtab-5}
\end{table}
\begin{table}[th!]
\centering \caption{Iteration number for Example \ref{UNCONEx-6}} 
\begin{tabular}{lllccc}
\hline
$n$ & SS1   &SS2&SS3&BB&NA\\ \hline
1000&110&39&26&29&31 \\ \hline
2000&111&39&26&30&28\\ \hline
5000&111&39&26&31&28\\ \hline
10000&112&39&26&31&28\\ \hline
20000&112&39&26&28&26\\ \hline
50000&115&41&29&35&32\\ \hline
\end{tabular}\label{UNCONtab-6}
\end{table}
\begin{table}[th!]
\centering \caption{Results for Example \ref{UNCONEx-8}} 
\begin{tabular}{lllc}
\hline
Methods&$k$&$\|g(x_k)\|$&CPU time\\ \hline
SS1&690&9.81e-07 &0.09 \\ \hline
SS2&46&6.40e-07&0.02 \\ \hline
SS3&37&8.02e-08&0.015\\ \hline
BB&102& 8.33e-07&0.02\\ \hline
NA&119&2.75e-07&0.02\\ \hline
\end{tabular}\label{UNCONtab-8}
\end{table}

\begin{table}[ht!]
\centering
\caption{Results for Example \ref{UNCONEx-9}}
\begin{tabular}{l|ccc|ccc}
\hline
Method & $k$ & CPU time & $\|g(x_k)\|$ & $k$ & CPU time & $\|g(x_k)\|$\\
\hline
& \multicolumn{3}{c|}{$n = 500$} & \multicolumn{3}{c}{$n = 1000$} \\
\hline
SS1 & 47 & 0.01 & 5.89e-07 & 49 & 0.01 & 7.56e-07\\
SS2 & 24 & 0.01 & 8.22e-07 & 26 & 0.01 & 3.16e-07\\
SS3 & 17 & 0.01 & 2.52e-07 & 18 & 0.02 & 1.78e-07\\
BB  & 5540 & 0.25 & 2.5e-07 & 14427 & 2.68 & 8.91e-07\\
ABB & 3457 & 0.17 & 6.89e-9 & 7842 & 1.53 & 8.13e-9 \\
ABBmin1 & 2514 & 0.14 & 7.18e-9 & 5326 & 1.09 & 6.98e-9 \\
ODH1 & 5851 & 0.32 & 8.67e-9 & 15060 & 3.08 & 8.07e-9 \\
CG  & 500 & 0.02 & 7.86e-08 & 1000 & 0.19 & 7.15e-08\\
\hline
& \multicolumn{3}{c|}{$n = 1500$} & \multicolumn{3}{c}{$n = 2000$} \\
\hline
SS1 & 51 & 0.03 & 5.56e-07 & 52 & 0.05 & 5.14e-07\\
SS2 & 27 & 0.04 & 2.74e-07 & 27 & 0.06 & 5.57e-07\\
SS3 & 18 & 0.02 & 5.53e-07 & 19 & 0.05 & 2.04e-07\\
BB  & 22706 & 23.48 & 8.07e-07 & 29538 & 52.83 & 8.82e-07\\
ABB & 11805 & 12.40 & 5.81e-9 & 19411 & 34.49 & 7.57e-9 \\
ABBmin1 & 7843 & 8.26 & 6.49e-9 & 10931 & 19.59 & 6.39e-9 \\
ODH1 & 20930 & 22.13 & 8.01e-9 & 31022 & 57.37 & 1.27e-7 \\
CG  & 1500 & 1.57 & 6.63e-8 & 2000 & 3.56 & 1.04e-8\\
\hline
\end{tabular}\label{UNCONtab-9}
\end{table}
\begin{table}[th!]
\centering \caption{Rate of convergence for Example \ref{UNCONEx-1}} 
\begin{tabular}{lccr}
\hline
\text{Method}    &$k$& $\|g(x_{k})\|$ &\text{ACOC}\\ \hline
SS1 &9&9.502e-23 & 2.00\\
SS2 &6& 2.1082e-18& 3.98\\
SS3 &5&2.6362e-14 & 5.76\\
\hline
\end{tabular}\label{UNCONtab-10}
\end{table}

In some cases, the obtained gradient system needed further simplification. For example, in \ref{UNCONEx-4}, the gradient system takes on a specific form
\begin{equation}\label{UNCONEx-st}
A\tilde{g}(x)=0,
\end{equation}
where  $A=diag\left(\dfrac{1}{10}, \dfrac{2}{10}, \ldots, \dfrac{n}{10}\right)^T$ is ill-conditional matrix. Thus, the solution of system \eqref{UNCONEx-st} has artificial difficulty. After multiplication by $A^{-1}$ the equation \eqref{UNCONEx-st} leads to
\begin{equation} \label{UNCONEx-st33}
\tilde{g}(x) =
\left\{
\begin{aligned}
\exp(x_1)& - 1, \\
\exp(x_2)& - 1, \\
\vdots& \\
\exp(x_n)& - 1,
\end{aligned}
\right.
\end{equation}
which equivalent to \eqref{UNCONEx-st} and is easy to solve. Upon careful comparison, it is evident that equation~\eqref{UNCONEx-st33} is identical to $g(x)$ in Example~\ref{UNCONEx-1} (refer to Table~\ref{UNCONtab-1}).
Essentially, comparing Table~\ref{UNCONtab-1} and Table~\ref{UNCONtab-4} shows that preconditioning by $A^{-1}$ results in fewer iterations.
\section{Conclusions}\label{UNC}
In this paper, we have proposed a class of super-schemes for solving nonlinear unconstrained optimization problems. The main contributions of this work can be summarized as follows:
\begin{itemize}
  \item We have developed high-order iterative methods (SS1, SS2, and SS3) that achieve local convergence orders of two, four, and six, respectively.
  \item The computational complexity of the proposed methods is comparable to existing gradient-based approaches, with a cost of $\mathcal{O}(n^2)$ per iteration.
  \item 
  The proposed approach introduces two novel strategies for selecting step-size parameters ensuring the global convergence of the methods.
  \item The methods rely on component-wise operations of vectors rather than matrix-by-matrix operations, enhancing their simplicity and efficiency.
  \item The proposed schemes are particularly suitable for large-scale and ill-conditioned problems.
\end{itemize}
Given these advantages, the proposed super-schemes provide an effective and practical framework for unconstrained optimization, offering new directions for both theoretical research and numerical applications in this field.

\end{document}